\documentclass[a4paper]{amsart}

\usepackage[english]{babel}
\usepackage[utf8]{inputenc}
\usepackage[T1]{fontenc}
\usepackage{amssymb,amsmath,amsthm,amsfonts}
\usepackage{graphicx}
\usepackage[usenames,dvipsnames]{xcolor}
\usepackage{cancel}
\usepackage[colorlinks]{hyperref}
\usepackage[normalem]{ulem}
\usepackage{dsfont}
\usepackage{enumitem}
\usepackage[showonlyrefs]{mathtools}
\mathtoolsset{showonlyrefs=true}
\newcommand{\N}{\mathbb{N}}

\newcommand{\R}{\mathbb{R}}

\newcommand{\shell}{Q}

\newcommand{\eps}{\varepsilon}

\newcommand{\Om}{\Omega}
\def\N{\mathbb{N}}

\def\F{\mathcal{F}}

\newtheorem{proposition}{Proposition}[section]
\newtheorem{theorem}[proposition]{Theorem}

\newtheorem{lemma}[proposition]{Lemma}
\theoremstyle{definition}

\newtheorem{remark}[proposition]{Remark}
\title{Nonexistence for screened Hartree energies}

\author{Dario Mazzoleni}
\address{Dipartimento di Matematica F. Casorati\\
		Universit\`a di Pavia\\
		Via Ferrata 5, 27100 Pavia, Italy}
\email{dario.mazzoleni@unipv.it}

\author{Riccardo Moraschi}
\address{Dipartimento di Matematica F. Casorati\\
		Universit\`a di Pavia\\
		Via Ferrata 5, 27100 Pavia, Italy}
\email{r.moraschi1@campus.unimib.it}

\author{Berardo Ruffini}
\address{Dipartimento di Matematica, Universit\`a di Bologna, Piazza di Porta S.Donato
5, 40126, Bologna, Italy}
\email{berardo.ruffini@unibo.it}

\begin{document}

\begin{abstract}
We develop an anti-concentration method for volume-constrained nonlocal shape optimization. Precisely, we apply it to prove the nonexistence of minimizers of an optimal design problem driven by a Hartree-type energy in the large mass regimes. The proof refines and extends the ideas from \cite{LuOtto}, where Lu and Otto dealt with nonexistence of minimizers for the Thomas-Fermi-Dirac-von Weizs\"acker energy. We also establish the asymptotic behavior of the associated isoperimetric profile of the energy. 
\end{abstract}

\maketitle
	
%\tableofcontents


\smallskip


\section{Introduction}\label{intro}
Consider the energy functional 
	\begin{equation}
		\label{eq:Eq}
		E_q(\Omega):=\inf_{u\in H^1_0(\Omega)} \left\{E_q(u,\Omega) \, : \, \int_\Omega u^2(x) dx = 1 \right\},    
	\end{equation}
    where
    \[
E_q(u,\Omega):= \int_{\Omega}|\nabla
		u(x)|^2 \, dx +{\frac{q}{2}}D(u^2,u^2)\qquad\text{and}\qquad   D(u^2,u^2):=  \int_{\R^3}\int_{\R^3}\frac{u^2(x)
			\, u^2(y)}{|x-y|}\,dx\,dy
    \]
    is the Coulomb energy of $u^2$, $q > 0$ is a fixed parameter, which we refer to as the \emph{charge}, and $\Omega \subset \mathbb R^3$ is an open set. We are interested in the shape optimization problem
    \begin{equation}\label{eq: mainminprob}
\min\left\{E_q(\Omega)\,:\,|\Omega|=|B_1| \right\},
    \end{equation}
    where $|\Omega|$ is the Lebesgue measure in $\R^3$ of $\Omega$ and $B_1$ is a generic unit ball.

If $q=0$  then $E_q$ is the first eigenvalue of the Dirichlet Laplacian and the related shape optimization problem \eqref{eq: mainminprob} leads to the Faber-Krahn inequality, stating that the first eigenvalue of the Dirichlet Laplacian is uniquely minimized by balls among sets of prescribed measure. In this sense, the Dirichlet energy appearing as the first term in the energy $E_q(\Omega)$ acts, once paired with the volume constraint, as a cohesive term. 
By contrast, in the limit case $q=+\infty$ the energy behaves as the reciprocal of the potential capacity, see \cite{landkof1972}, and acts as a scattering term. This competition suggests how the magnitude of the charge $q$ plays a role in order to establish well posedness -or lack thereof- for the problem.

The energy functional in \eqref{eq:Eq} derives from a famous model introduced by Hartree in 1928, \cite{hartree}, in quantum mechanics. The original model was formulated in the whole space $\R^3$ to describe the behavior of electrons and underwent several transformations, leading in particular to the restricted Hartree equation, 
\[
-\Delta u+Vu+\Phi[u^2]u=\varepsilon_q u
\]
with $u:\R^3\to\R$. Here $V$ is an external potential that usually acts as a confining term, while
\[
\Phi[u^2](x)=u^2\star\frac{1}{|\cdot|}(x)=\int_{\R^3}\frac{u^2(x-y)}{|y|}\,dy
\]
is the potential of $u^2$ in $\R^3$. We refer to \cite{lions} for a precise mathematical treatment of this problem in the whole space $\R^3$.

We note that problem \eqref{eq: mainminprob} can be restated, by a simple scaling argument (see Lemma \ref{le:equivalence}) as follows: 
\begin{equation}\label{eq:mainnonex}
		\min\left\{ \mathcal F_m(\Omega)\,:\, \Omega\subset
		\R^3,\;\text{open, with }|\Omega|=|B_1|\, \right\} 
	\end{equation}
	where
	\begin{equation}
		\label{eq:Fm}
		\mathcal F_m(\Omega)=\min_{ u\in
			H^1_0(\Omega)}\left\{\int_\Omega|\nabla u|^2 dx +\frac12
		D( u^2, u^2)\,:\, 
		\int_\Omega u^2 dx =m\right\}=\min_{ u\in
			H^1_0(\Omega)}\left\{E_1(u,\Omega)\,:\, 
		\int_\Omega u^2 dx =m\right\}.
	\end{equation} 
In particular we have that
\[
E_q(\Omega)=q^{-1}\F_q(\Omega).
\]
The mathematical advantage of the latter formulation is that we can fix the functional by tuning the $L^2$ mass of the ground states of each set. This also underlines the capacitary nature of the energy for $m$ large: formally neglecting the Dirichlet term, given a set $\Omega$, we ask for the maximal amount of charge $m$ that such a set can contain.  

Notice also  that our energy is exactly the Hartree reduced energy where the confining potential $V$ is replaced by an abrupt screening one, that is, where we formally impose 
\[
V=+\infty\,\chi_{\R^3\setminus \Omega}=
\begin{cases}
    +\infty\qquad&\text{ on $\R^3\setminus\Omega$},\\
    0\qquad&\text{ on $\Omega$}.
\end{cases}
\]
Let us note that these kinds of potentials are rigorously introduced in the shape optimization context as capacitary measures; moreover  such \emph{bang-bang} screening potentials naturally appear as limiting cases in optimization for Schr\"odinger operators, see for instance \cite[Chapter 9]{hen17} and the references therein.

In our setting, the  shape optimization problem \eqref{eq: mainminprob} was introduced by the first and the third author together with C. B. Muratov in \cite{MazzoleniMuratovRuffini} in order to describe the optimal shape of a superconducting island $\Omega$ embedded into an insulator. In particular, the optimal set $\Omega$ in \eqref{eq:mainnonex} (or \eqref{eq: mainminprob}) is the optimal shape that an insulator may assume in order to maximize the number of (super-)charges in it. 
We refer to \cite{MazzoleniMuratovRuffini} and the references therein for a more precise introduction of the physical interpretation and background of such a model. 

In addition to  its physical interpretation, from a variational point of view the energy \eqref{eq:Eq} is a natural counterpart of the famous Gamow liquid drop model, where the shape energy up to renormalization of the physical constants  is given by
\[
G(\Omega)=\int_{\R^3}|\nabla \chi_\Omega|\,dx+\eps\int_{\R^3}\int_{\R^3}\frac{\chi_\Omega(x)\chi_\Omega(y)}{|x-y|}\,dxdy
%=P(\Omega)+\varepsilon\int_\Omega\int_\Omega\frac{dx\,dy}{|x-y|}
=P(\Omega)+\varepsilon D(\chi_\Omega,\chi_\Omega), \qquad \Omega\subset\R^3, \,\,|\Omega|=1,
\]
where $P(\Omega)$ is the distributional (De Giorgi) perimeter of $\Omega$, see \cite{Maggibook} for definitions and properties on that.
In fact, the two functionals share the same competition of a cohesive term and a repulsive one. In particular it was first shown in \cite{KM1,KM2} that if $\varepsilon\ll1$ then balls are the unique volume-constrained minimizers of $G$, while if $\eps\gg 1$ no minimizer exists. More precisely, for the Gamow energy it has been conjectured in \cite{ChoksiPeletier} that there exists a unique threshold $\varepsilon^*>0$ such that for $\varepsilon\le\varepsilon^* $ the unit ball is the unique minimizer while for $\varepsilon>\varepsilon^*$ no minimizer is attained. Although this conjecture is still open, there are some partial results  in its support. We refer to the introduction of the recent paper \cite{BMP25} for a clear and exhaustive state of the art on such a conjecture.

Quite interestingly, the existence theory for these problems is somewhat better understood than the nonexistence theory. The general idea is that whenever the cohesive term enjoys strong enough stability properties, the ball, provided it is a critical point, must be a \emph{local minimizer}, in suitably chosen topologies. Hence the problem reduces to show that any minimizer is in fact a suitable small perturbation of the ball. Such a strategy was employed in \cite{MazzoleniMuratovRuffini}, where it is shown that for small values of $q$, minimizer exists for \eqref{eq: mainminprob}; they are $C^{2,\alpha}-$regular sets and they approximate a ball as $q\to0$.  

We now focus on the nonexistence issue, which was left unsolved in~\cite[Conjecture 6.1]{MazzoleniMuratovRuffini}.

A standard strategy to obtain such nonexistence results is to show that the isoperimetric profile of the energy is strictly concave with respect to a leading parameter. This implies that the functional favors scattering of the mass and this leads to nonexistence. For perturbations of the perimeter, as in the Gamow energy, a standard technique is to reason by contradiction and suitably split any minimizer so that the gain in energy of the repulsive part is higher than the loss of the cohesive one, see for instance \cite{FFMMM15,FKN16,FN21,FNV18_CPAM}.  

This article deals precisely with the nonexistence  for the screened Hartree functional \eqref{eq: mainminprob}.
Namely, our main result is the following.

\begin{theorem}\label{thm:maintheorem}
			There exists a universal constant $\overline q>0$ such that if $q>\overline q$, \eqref{eq: mainminprob} has no solution.
\end{theorem}

The proof is inspired by an elegant argument proposed in~\cite{LuOtto} to show nonexistence, for $m$ large, of ground states of the Thomas-Fermi-Dirac-von Weizs\"acker (TFDW) energy with prescribed mass $m$, namely
\begin{equation}\label{eq:luotto}
    \inf\left\{ \int_{\R^3}\left(|\nabla u|^2+u^{10/3}-u^{8/3}\right)\,dx+\frac12 D(u^2,u^2)\,:\, \int_{\R^3} u^2\,dx=m \right\}.
\end{equation}
To put our proof into perspective, and in order to clarify it, it is worth outlining the key points of the proof for such functional.
\begin{itemize}
\item[-] First, build a suitable competitor by means of a partition of unity in order to split a minimizer $u$ into more components and send such components of $u$ far  apart from each other. This replaces the abrupt cutting technique of the perimeter case. 
\item[-] Then test the minimality of the functional against such a competitor and perform the gain versus loss analysis with the following ideas:
\begin{enumerate}
\item the gradient penalization is localized in the overlapping region of the partition of unity.
\item the Coulombic gain is governed by the nonlocal interaction with the distant mass. 
\end{enumerate}
\item[-] Hence deduce an energy inequality that relates the local $L^2$-mass (situated near the splitting locus) to the mass distributed in the far-field. 
\item[-] Next, Lu and Otto show via a sort of discrete-ODE argument that the existence of a minimizer would require a specific concentration principle: if a small amount of mass is present locally, the ratio between the local and distant $L^2$-norms must be large, thus the mass must accumulate. 
\item[-] The argument ends noting that such a  concentration phenomenon, however, is energetically hostile in the large-mass regime, where scattering is expected to predominate, thereby leading to the nonexistence of a global minimizer.
\end{itemize}
In our context  this strategy fails, as it is, for two reasons: first and foremost, it requires the function $u$ to be supported -in principle- in the whole space, so that there is no penalization in terms of the support. Second, the kinetic penalization in \eqref{eq:luotto},
\begin{equation}\label{eq:kinetic}
\int_{\R^3}u^{10/3}-u^{8/3}\,dx,
\end{equation}
forces any ground state $u$ to be uniformly bounded in $L^\infty$. This imposes the mass of $u$ to spread and plays a crucial role in the final contradiction argument by Lu and Otto.

In our setting the measure constraint imposed on the support of optimizers prevents both these simplifications (even adding the kinetic terms~\eqref{eq:kinetic}), leading to a quite different proof where the analytic reasoning must be implemented by a purely geometric argument.

We summarize here the main additional difficulties and differences.
\begin{itemize}
\item[-] A standard partition of unity cannot be directly applied to construct a valid competitor, as it would artificially enlarge the measure of the test function's support. Such a term might be properly bounded \emph{only} by knowing its geometric features quite precisely, which appears impossible a priori. Consequently, we are forced to separate the components of the putative minimizer using strictly nonoverlapping cut-off functions. This ``hard'' separation induces a critical loss of information regarding the $H^1$-mass located within the separation layer. While this missing mass has a negligible impact on the nonlocal Coulomb interaction, it severely compromises the local gradient estimates, which no longer benefit from the properties guaranteed by a standard partition of unity.
\item[-] To overcome this issue, we must deeply refine our gradient bounds. The key idea relies on establishing a Caccioppoli-type estimate inside the layer (i.e. the zone ``forgotten'' by our cut-off functions), which allows us to bound the Dirichlet energy across the separation layer despite the missing information. 
\item[-] Nevertheless, both gradient and Coulomb estimates lead to a successful contradiction only if the amount of mass trapped in the separation layer is sufficiently small; otherwise, an excessive loss of $L^2$-mass in this region would leave us without any kind of usable information. 
\item[-] This analysis leads to a dichotomy: either we reach the desired contradiction, or almost all of the mass concentrates in a thin layer far away from the origin (we can actually bound from below its distance from the origin by $\sqrt m$, as $m\gg1$). This is the content of Theorem \ref{thm:nomin}, which contains the technical core of the proof of Theorem \ref{thm:maintheorem}.  

\vspace{.2cm}

From a physically flavored heuristic, we are saying that if the charge $m$ is very large, then existence can occur only if such a charge disposes itself in a thin layer far from the origin, which escapes to infinity as $m$ diverges. This is not sufficient to conclude from a purely energetic argument, as in \cite{LuOtto}.

\vspace{.2cm}

\item[-] Instead, to rigorously rule out such an anomalous mass concentration on the splitting zone (the thin layer), we introduce a novel geometric argument. Instead of the spherical splitting employed in the framework of \cite{LuOtto}, we perform the separation using cut-off functions supported on a cube and its complement\footnote{Any convex body not too symmetric works as well, but the cube is the easiest one to be exploited.}. By exploiting the lack of rotational invariance of the cube, we repeat the concentration argument on three different cubes, with different orientations. By a  geometric argument, performed in Appendix \ref{approt2}, we show that it is possible to choose the cubes so that much of the mass must be concentrated on \emph{uniformly bounded} subsets obtained as intersections of the three selected cubes. This leads to an excess of Coulomb energy, violating the upper bounds of the functional which we can deduce by an explicit competitor (a mist of droplets escaping to infinity).
\end{itemize}

\begin{remark}
Albeit the above strategy is applied directly to problem~\eqref{eq:Eq}, we stress that it is a general anti-concentration method, which can be employed for a large class of spectral problems where a double constraint on the mass and on the support is present. 
\end{remark}

\vspace{.5cm}
By Theorem \ref{thm:maintheorem} we know that $\mathcal F_m$ does not admit volume constrained minimizers. Nevertheless by a comparison argument, which we perform in Appendix \ref{app2}, it is possible to compute an explicit asymptotic upper bound for the energy, depending on $m$. A second result of this paper is to show that such an upper bound is optimal,  up to multiplicative constant. Precisely we give the following estimate on  the isoperimetric profile of the energy as $m$ diverges, that is, letting

\[
\mathcal I(m)=\inf \left\{\mathcal F_m(\Omega)\,:\, \Omega\subset \R^3,\;\text{open, with }|\Omega|=|B_1|\right\},
\]
we show the following theorem.
\begin{theorem}\label{thm:asymptotic profile}
  There exists a  universal constant  $C\ge1$ such that if $m\ge 1$ then
\begin{equation}\label{eq:isoprofile}
\frac{m^\frac32}{C}\le \mathcal I(m)\le Cm^\frac32,
\end{equation}
in particular, \[
\frac{m^\frac12}{C}\le \inf\{E_m(\Omega) \,:\, |\Omega|=|B_1|\}\le Cm^\frac12.
\]
\end{theorem}

Let us spend a few words on the proof of this result as it is due to a quite simple idea. As mentioned above, the upper bound in \eqref{eq:isoprofile} is easier: one just tests the energy on a disjoint union of equal balls escaping to infinity, and then optimizes in terms of the radius of such balls. This is done in Appendix~\ref{app2}. 

The lower bound is a little bit less direct: assuming that $u$ is an optimal function for $\F_m(\Omega)$, we have that
\[
\F_m(\Omega)=\int_\Omega |\nabla u|^2+\frac{1}{2}D(u^2,u^2)=m\int_\Omega |\nabla (u/\sqrt{m})|^2+\frac{m^2}{2}D((u/\sqrt{m})^2,(u/\sqrt{m})^2)\ge m\lambda_1(\Omega)+\frac{m^2}{2\rm cap(\Omega)},
\]
where $\lambda_1(\Omega)$ is the first Dirichlet eigenvalue of $\Omega$ while ${\rm cap}(\Omega)$ is its potential capacity, see, e.g., \cite{landkof1972}. 
So, by the arithmetic-geometric inequality, one obtains 
\[
\F_m(\Omega)\ge \sqrt{8}m^{3/2} \sqrt{\frac{\lambda_1(\Omega)}{{\rm cap}(\Omega)}}.
\]
This suggests that the lower bound is sharp, up to a multiplicative constant, but unfortunately the argument is not conclusive as the shape functional $\Omega\mapsto \frac{\lambda_1(\Omega)}{{\rm cap}(\Omega)}$ is not bounded from below by a strictly positive constant even if the volume of $\Omega$ is fixed. One counterexample is a ball attached to a thin tentacle of length $L$, for which the energy of $\lambda_1(\Omega){\rm cap}(\Omega)^{-1} $ is proportional to $L^{-1}$, converging to $0$ as $L\to+\infty$. Nonetheless, by means of a heuristic argument one can see that the optimal thickness (i.e. the diameter of its section) of the tentacle should be of order $m^{-1/4}$. Hence,  to obtain a formal proof we reason as follows: we divide the space into cubes of side-length $m^{-1/4}$ and select those \emph{quite filled}, where the capacitary term $D(\cdot,\cdot)$ is more relevant, and those \emph{quite empty}, where the spectral term given by the Dirichlet energy is expected to be large, as there is enough boundary of $\Omega$ to support a Poincar\'e type inequality. Computing separately these two cases leads to the desired lower bound in \eqref{eq:isoprofile}.

\subsection{Organization of the paper}

The paper is organized as follows. Section \ref{sec:nonexistence} is  devoted to the proof of Theorem \ref{thm:maintheorem}. As explained above, a main part of it is the proof of Theorem \ref{thm:nomin}, which covers much of the section. Section \ref{sec:asymptotic} is devoted to the lower bound for the proof of Theorem \ref{thm:asymptotic profile}. The upper bound is stated and proved in Appendix \ref{app2}, which is useful also for the nonexistence result in Theorem \ref{thm:maintheorem}. In Appendix \ref{approt2} we show the existence of (many) good rotations in order to minimize the mutual intersection of thin  cube-layers. Finally, Appendix~\ref{app:poinc} we prove a Poncar\'e-type inequality needed in the proof of Theorem~\ref{thm:asymptotic profile}.

\section{Proof of Theorem~\ref{thm:nomin}}\label{sec:nonexistence}

In this section we show that problem \eqref{eq:mainnonex} does not
admit minimizers for $m$ large enough, namely we prove Theorem~\ref{thm:maintheorem}.   
        Throughout the paper, we tacitly extend any function $u\in H^1_0(\Omega)$ by zero outside $\Omega$ and we adopt the following notation: for
       $\Omega \subset \R^3$ an open set of finite measure and
        $u\in H^1_0(\Omega)$ denoting by $\star$ the usual convolution, we let
	\begin{equation}\label{eq:defvuhu}
		\Phi[u^2](x)={ u^2 \star \frac{1}{|\cdot|}(x) } =\int_{\R^3}\frac{u^2(y)}{|x-y|}\,dy=\int_{\Omega}\frac{u^2(y)}{|x-y|}\,dy.
              \end{equation}
Notice that $\Phi[u^2] \in W^{2,3}_\mathrm{loc}(\R^3)$ by \cite[Theorem 9.9]{GilbargTrudinger} and Sobolev embedding.  We begin with some simple remarks.
	\begin{proposition}\label{prop:existence optimal u}
          Let $\Omega\subset \R^3$ be an open set of finite measure and
          let $m>0$.  Then the minimization problem \eqref{eq:Fm}	admits a solution with constant sign (nonnegative, without loss of generality). 
	\end{proposition}	%
	\begin{proof}
		Let $(u_n)_n$  be a minimizing sequence for the energy. Since all the terms in the definition of $E_1(u,\Omega)$ are positive, we infer that $(u_n)_n$ is bounded in $H^1_0(\Omega)$. Then up to passing to a subsequence, $(u_n)_n$ converges weakly in $H^1_0(\Omega)$ and strongly in $L^2(\Omega)$ to some function $u\in H^1_0(\Omega)$. In particular the $L^2-$convergence implies
		that $\|u\|_{L^2(\Omega)}=m$.  Observe that in the case of finite measure unbounded set, the Sobolev embedding holds by \cite{BucurButtazzo}. By lower semicontinuity with respect to the weak convergence, we have that
		\[
		\int |\nabla u|^2\,dx\le \liminf_{n\to+\infty} \int |\nabla u_n|^2\,dx
		\]
		and, by Fatou's Lemma 
		\[
D(u^2,u^2)=\iint_{\Omega\times\Omega}\frac{u^2(x)u^2(y)}{|x-y|}\,dxdy\le \liminf_{n\to+\infty} D(u_n^2,u_n^2).
		\]
		Hence, the energy is lower semicontinuous
		and so $u$ is a minimizer.
		The fact that $u$ can be chosen of constant sign follows since $u$ and $|u|$ yield the same energy value.
              \end{proof}

\begin{lemma}\label{le:equivalence}
	Let $q=m>0$. Then problems~\eqref{eq: mainminprob} and \eqref{eq:mainnonex} are equivalent.	
\end{lemma}
\begin{proof}
Let us fix $\Omega\subset \R^3$ an open set of measure $|B_1|$
	and let $ u\in H^1_0(\Omega)$ with $\int_\Omega  u^2 dx =m$ be
	a function attaining $\mathcal F_m(\Omega)$ by Proposition \ref{prop:existence optimal u}.  Then, since
	$\int_\Omega (m^{-1/2} u)^2 dx =1$,
	\[ E_m(\Omega)=m^{-1}\mathcal F_m(\Omega). \qedhere
	\]
\end{proof}
Let $m>0$ be such that problem \eqref{eq:mainnonex} admits a minimum. We denote by $\mathcal{U}_m$ the set of all optimal functions $u$ associated with some minimizing domain $\Omega$. More precisely,
\begin{equation}\label{eq:Um}
\mathcal{U}_m=\left\{u\in H^1(\R^3)\,:\,\exists\Omega\subset\R^3,\,\text{$\Omega$ is a solution of \eqref{eq:mainnonex} and $u$ is a solution for \eqref{eq:Fm}}\,\right\},
\end{equation}
i.e., $u \in \mathcal{U}_m$ if there exists a domain $\Omega$, which is a minimizer for problem \eqref{eq:mainnonex} at the given $m$, such that $u$ is an optimal state function for $\Omega$.

We now state and show the cornerstone result we need in order to prove Theorem \ref{thm:maintheorem}.
\begin{theorem}\label{thm:nomin}
			Let $\shell$ be a convex bounded open set. Then there exists  $\overline m(\shell)\geq1$ such that for each $m\geq \overline{m}$, exactly one of the following claims hold true: 
			\begin{enumerate}[label=\roman*)]
				\item  problem \eqref{eq:mainnonex} has no solution for such $m$.	
				\item  problem \eqref{eq:mainnonex} admits a solution and
                \begin{equation}\label{eq:hpassurda}
                 \sup_{x\in \R^3}\sup_{r>0} \bigg\{ \frac{1}{m}\int_{\shell_{r+1}(x)\setminus \shell_r(x)} u_{m}^2\,dz \bigg\} > \frac{9}{10}, \quad \forall\; u_m\in \mathcal{U}_m.
                \end{equation}
                where $\shell_r(x)=\{ry+x: y\in \shell\}$.
			\end{enumerate} 
		\end{theorem}
\begin{proof}[Proof of Theorem \ref{thm:nomin}]
We set in this proof, with a slight abuse of notation,
\[
\F( u)=\int_\Omega|\nabla u|^2\, dx+\frac12 D( u^2, u^2),
\]
and we indicate with $C$ a positive constant depending only on $\shell$, which may increase from line to line.
 
Let $m>0$. Suppose \textit{ii)} does not hold. If \eqref{eq:mainnonex} does not admit a solution, \textit{i)} trivially holds. Otherwise there exists $u\in \mathcal{U}_m$ such that 
\begin{equation}\label{eq:striscia con poca massa}
    \sup_{x\in \R^3}\sup_{r>0} \bigg\{ \frac{1}{m}\int_{\shell_{r+1}(x)\setminus \shell_r(x)} u^2\,dz \bigg\} \leq \frac{9}{10}.
\end{equation}
The rest of the proof is devoted to showing that this can happen only if $m\leq \overline m$ for a certain $\overline m(Q)$.
Let us denote by $\Omega$ a minimizer of \eqref{eq:mainnonex} for which $u$ is an optimal function. Then $u\geq 0$ in $\Omega$ and
we extend $ u$ by zero	outside of $\Omega$ and define, for $x\in\R^3$ and $r>0$,
\[
V(x,r)=\int_{\shell_r(x)} u^2{(y)\, dy}.
\]
We divide  the proof into several steps.
					
		\vspace{.3cm} {\bf Step 0: The Euler-Lagrange equation for $u$} 
		
			 We now state the Euler-Lagrange equation solved by a solution $u$ of $\F_m(\Omega)$, which will be useful in the following proof. 
%By appendix \ref{appendix} $q$ and $m$ are linearly comparable. 
By a simple scaling argument, if $\widetilde u$ is a solution of problem \eqref{eq: mainminprob}, then $u=m^{1/2}\widetilde u$ is a minimizer of \eqref{eq:Fm}, see Lemma~\ref{le:equivalence}. Moreover, by direct computation (see~\cite[Lemma~2.4]{MazzoleniMuratovRuffini}),  the first variation of the energy implies that $\widetilde u\in H^1_0(\Omega)$ satisfies weakly in $H^1_0(\Omega)$
\[
-\Delta\widetilde u=\alpha(x)\widetilde u,
\]
where 
		\[
		\alpha(x)=\int_\Omega|\nabla \widetilde u|^2\,dx+mD(\widetilde u^2,\widetilde u^2)-m\Phi[\tilde{u}^2](x).
		\]
		Hence $u\in H^1_0(\Omega)$ solves weakly in $H^1_0(\Omega)$ 
		\begin{equation}\label{eq:ELm}
		-\Delta u=\beta(x)u
		\end{equation}
		where (see Lemma~\ref{le:equivalence})
		\[
		\beta(x)=\left[\frac{1}{m}\left(\int_\Omega|\nabla  u|^2\,dx+D( u^2, u^2)-m\Phi[u^2](x)\right)\right].
		\]
		It is useful to notice that, by the minimality of $\Omega$, there holds  
		\begin{equation}\label{eq:boundC}
		\beta(x)\le m^{-1}\F_m(\Omega)\le C \sqrt{m}, 
\end{equation}		
		for some universal constant $C$, and for $m$ large enough, thanks to estimate~\eqref{eq:boundm32} of Appendix~\ref{app2}.
		
		\vspace{.3cm} {\bf Step 1: A discrete ODE for $V$} 

 We show, following ideas from \cite{LuOtto}, that $V(x,\cdot)$ satisfies a suitable inequality. Precisely we claim that there exists $\gamma=\gamma(\shell)>0$ such that for all $x\in \R^3$, $r>0$, $t\in (0,1]$ and $R>r+t$, the following discrete-ODE type inequality holds:
					\begin{equation}\label{eq:V}
				\Big(V(x,r+t)-V(x,r)\Big)\geq \frac{\gamma}{(\tfrac{1}{t^2}+\sqrt{m})R}\Big(V(x,R)-V(x,r+t)\Big)V(x,r).
			\end{equation}
This step is the technical core of the proof.

\vspace{.5cm}

Let $f_1,f_2\in C^\infty(\R^3,[0,1])$ be cutoff functions such that 
\begin{align}
				&f_1=1 \qquad \text{in $\shell_{r+\tfrac{t}{4}}(x)$}, \qquad &&\text{and	$f_1 = 0$ outside $\shell_{r+\tfrac{t}{2}}(x)$},\\
				&f_2=1 \qquad \text{in $\R^3\setminus \shell_{r+\tfrac{3t}{4}}(x)$}, \qquad &&\text{and $f_2 = 0$ inside $\shell_{r+\tfrac{t}{2}}(x)$}.
\end{align}
			and $|\nabla f_{1}|+|\nabla f_{2}|\le \tfrac{C}{t}$ for some $C > 0$ depending on $\shell$. 
We can now define $u_i=f_i u\in H^1_0(\Omega)$, for $i=1,2$.
        \vspace{0.3cm}

{\bf Step 1.1: Comparing the energies.}
		We define, for $L>1$,
		\[
		\psi_L(z)=C_L\Big( u_1(z)+ u_2(z+L \nu)\Big),\quad z\in \R^3,
		\]
		where $\nu\in\R^3\setminus\{0\}$ is a nonzero vector and $C_L>0$ is a constant chosen so that
		$\|\psi_L\|_{L^2(\R^3)}^{2}=m$.  Let $\overline C$ be the positive constant such that $m=\int_{\R^3}(\overline C u_1)^2+(\overline C u_2)^2\,dx$.  We now aim to show that 
        \begin{enumerate}
        \item[$a)$]$C_L\to\overline C$ as $L\to+\infty$, 
        \item[$b)$]It holds, for $m$ large enough,
			\begin{equation}\label{eq:ee4}
			\overline C^2\leq 1+\frac{10}{m}\int_{\shell_{r+t}(x)\setminus \shell_r(x)} u^2\,dz.
		\end{equation} 
    \end{enumerate}
    Indeed
            \begin{align}\label{eq: ee5}
m=C_L^2\bigg(\int_{\R^3}u_1^2+\int_{\R^3}u_2^2+2\int_{\R^3}u_1(x)u_2(x+L\nu)dx\bigg)=\int_{\R^3}(C_Lu_1)^2+\int_{\R^3}(C_Lu_2)^2+2C_L^2\int_{\R^3}u_1(x)u_2(x+L\nu),
            \end{align}
            so to show the convergence of $C_L$ to $\overline{C}$, it is enough to show that the last term on the right-hand side of the previous formula converges to $0$ as $L$ diverges. Let $\varepsilon>0$, then by density, there exists $v_1,v_2\in C_c^{\infty}(\R^3)$ such that $\|u_i-v_i\|_{L^2(\R^3)}\leq \varepsilon$ for $i\in\{1,2\}$. Thus
            \begin{align*}
                \bigg| \int_{\R^3}u_1(x)u_2(x+L\nu)dx\bigg) \bigg|\leq& \int_{\R^3}|(u_1(x)-v_1(x))u_2(x+L\nu)|dx+ \int_{\R^3}|v_1(x)(u_2-v_2)(x+L\nu)|dx\\&+\int_{\R^3}|v_1(x)v_2(x+Lv)|dx\\ \leq&\eps \sqrt{m}+2\eps \sqrt{m},
            \end{align*}
            where the last inequality holds by Cauchy-Schwarz, the fact that $\int_{\R^3}v_1(x)v_2(x+Lv)dx=0$ for $L$ big enough {and $||v_i||_{L^2(\R^3)}\leq2\sqrt{m}$ for $i\in\{1,2\}$}. So point $a)$ is proved.
            
            Let us now prove point $b)$, that is, \eqref{eq:ee4}. We have
            \[
			\frac{m}{\overline C^2}=\int_{\R^3} u_1^2+ u_2^2\,dz=m-\int_{\shell_{r+t}(x)\setminus \shell_r(x)}u^2-(f_1^2 u^2+f_2^2 u^2)\,dz \geq 0,
			\]			
			so we deduce,
			\[
			\overline C^2\left(1-\frac{1}{m}\int_{\shell_{r+t}(x)\setminus \shell_r(x)} u^2\,dz\right)\leq 1.
			\]

  Recalling \eqref{eq:striscia con poca massa}\footnote{The contradiction assumption~\eqref{eq:striscia con poca massa} is used only here, to have a uniform bound on the renormalizing constant $\overline C$.}, we obtain 
			\begin{align*}\label{eq:Cbar}
				\overline C^2\leq& \frac{1}{1-\frac{1}{m}\int_{\shell_{r+t}(x)\setminus \shell_r(x)} u^2\,dz}= \sum_{i=0}^{\infty} \bigg(\frac{1}{m}\int_{\shell_{r+t}(x)\setminus \shell_r(x)} u^2\,dz\bigg)^i \\ = &1+\sum_{i=0}^{\infty} \bigg(\frac{1}{m}\int_{\shell_{r+t}(x)\setminus \shell_r(x)} u^2\,dz\bigg)^i\bigg(\frac{1}{m}\int_{\shell_{r+t}(x)\setminus \shell_r(x)} u^2\,dz\bigg) 
                \\ \leq& 1+\sum_{i=0}^\infty\left(\frac{9}{10}\right)^i \bigg(\frac{1}{m}\int_{\shell_{r+t}(x)\setminus \shell_r(x)} u^2\,dz\bigg)\\
                =&1+\frac{10}{m}\int_{\shell_{r+t}(x)\setminus \shell_r(x)} u^2\,dz,
			\end{align*} 
            since $\sum_{i\ge0}\left(\frac{9}{10}\right)^i=10$, so also point $b)$ is proved.

            \noindent
		Let us now define $\Omega_L=\{\psi_L>0\}$ and notice that 
			\[
			|\Omega_L|
			\le|\Omega|=|B_1|,
			\]
hence, since $\Om\mapsto\mathcal F(\Om)$ is decreasing
		with respect to set inclusion (as the minimization in its definition
		is done under Dirichlet boundary condition), there holds, for any
		open set $A_L$ of measure ${|B_1|}$ and containing
		$\Omega_L$, that
		\begin{align}
			\F( u)&=\mathcal F_m(\Omega)\\
			&\le \F_m(A_L),\qquad&&\text{by minimality of $\Omega$}\\
			&\le \F_m(\Omega_L)\qquad &&\text{as $\Omega_L\subset A_L$}\\
			&\le \F(\psi_L)\qquad &&\text{as $\psi_L\in H^1_0(\Omega_L)$}.
		\end{align}
				Arguing as in point $a)$ above,	one sees that 
			\[
			\F(\psi_L)\to \F(\overline C u_1)+\F(\overline C u_2),\qquad \text{as $L\to+\infty$},
			\]
			so that 
			\begin{equation}\label{eq:subadditivity}
				\F(\overline C u_1)+\F(\overline C u_2)-\F( u)\ge 0.
			\end{equation}

{\bf Step 1.2: Analysis of the Dirichlet terms of~\eqref{eq:subadditivity} through a Caccioppoli inequality.}
		At this point a direct computation gives   
		\begin{equation}\label{eq:ee1}
			\begin{split}
				|\nabla  u_1|^2+|\nabla  u_2|^2-|\nabla  u|^2&\leq u^2(|\nabla f_1|^2+|\nabla f_2|^2)+2(f_1u\nabla f_1\cdot \nabla u+f_2u\nabla f_2  \cdot\nabla u).
			\end{split}
		\end{equation}
The first term in the right-hand side of \eqref{eq:ee1} satisfies, for some  constant $C=C(Q)$  
			\[
			u^2(|\nabla f_1|^2+|\nabla f_2|^2)\leq \frac{C}{t^2} u^2 \chi_{\shell_{r+t}(x)\setminus \shell_r(x)},
			\]
		since  $f_1$ and $f_2$ are constant outside the strip $\shell_{r+\tfrac{3t}{4}}(x)\setminus \shell_{r+\tfrac{t}{4}}(x)$ and $|\nabla f_{1}|+|\nabla f_{2}|\le \tfrac{C}{t}$. Then the previous inequality yields 
			\begin{equation}\label{eq:ee2}
				\int_{\Omega}u^2(|\nabla f_1|^2+|\nabla f_2|^2)\,dz\le \frac{C}{t^2}\int_{\shell_{r+t}(x)\setminus \shell_r(x)} u^2\,dz = \frac{C}{t^2}\left(V(r+t)-V(r)\right).
			\end{equation} 
			
		We now deal with the second term in the right-hand side of \eqref{eq:ee1}. 
		We claim that  
		\begin{equation}\label{eq:ee3}
\int_\Omega f_1u\nabla f_1\cdot \nabla u+f_2u\nabla f_2  \cdot\nabla u\,dz \le C\left(\frac{1}{t^2}+\sqrt{m}\right)(V(r+t)-V(r)).
		\end{equation}
First, $f_1$ and $f_2$ are constant outside the strip $\shell_{r+\tfrac{3t}{4}}(x)\setminus \shell_{r+\tfrac{t}{4}}(x)$, so  their gradients are zero.
Then, by the definition of $f_1$ and $f_2$ and by  Young's inequality one obtains
		\[
\begin{split}
		\int_{\shell_{r+\frac{3t}{4}}(x)\setminus \shell_{r+\frac{t}{4}}(x)}&(f_1u\nabla f_1\cdot \nabla u+f_2u\nabla f_2  \cdot\nabla u)\,dz\\
		&\le\frac12\int_{\shell_{r+\frac{3t}{4}}(x)\setminus \shell_{r+\frac{t}{4}}(x)}(f_1^2+f_2^2)|\nabla u|^2\,dz+\frac12\int_{\shell_{r+\frac{3t}{4}}(x)\setminus \shell_{r+\frac{t}{4}}(x)}u^2\Big(|\nabla f_1|^2+|\nabla f_2|^2\Big)\,dz\\
		&\leq  C\left(\int_{\shell_{r+\frac{3t}{4}}(x)\setminus \shell_{r+\frac{t}{4}}(x)}  |\nabla u|^2\,dz+\frac{1}{t^2}\int_{\shell_{r+\frac{3t}{4}}(x)\setminus \shell_{r+\frac{t}{4}}(x)}  u^2\,dz\right).
		\end{split}
		\]
      Hence showing  \eqref{eq:ee3} reduces to proving the following Caccioppoli-type inequality: 
		\begin{equation}\label{eq:caccioppoli}
		\int_{\shell_{r+\frac{3t}{4}}(x)\setminus \shell_{r+\frac{t}{4}}(x)}  |\nabla u|^2\,dx\le C\left(\frac{1}{t^2}+\sqrt{m}\right)\left(V(r+t)-V(r)\right).
		\end{equation}
		For this purpose, let $\xi\in C^\infty_c(\shell_{r+t}(x)\setminus \shell_r(x))$ such that $\xi=1$ in $\shell_{r+\tfrac{3t}{4}}(x)\setminus \shell_{r+\tfrac{t}{4}}(x)$ and $|\nabla \xi|\leq \frac{C}{t}$ with $C=C(Q)$. Testing \eqref{eq:ELm} with $\eta=\xi^2 u$ yields, using estimate~\eqref{eq:boundC}, to
		\[
		\begin{aligned}
		\int_\Omega |\nabla u|^2\xi^2\,dz&=\int_{\shell_{r+t}(x)\setminus \shell_r(x)}\beta(x)u^2\xi^2\,dz-\int_{\shell_{r+t}(x)\setminus \shell_r(x)}2u\xi\langle\nabla u,\nabla\xi\rangle\,dz\\
		&\le C\sqrt{m}\int_{\shell_{r+t}(x)\setminus \shell_r(x)}u^2\,dz+C\int_{\shell_{r+t}(x)\setminus \shell_r(x)}u\xi |\nabla u||\nabla \xi|\,dz\\
		&\le C\sqrt{m}\int_{\shell_{r+t}(x)\setminus \shell_r(x)}u^2\,dz+\frac12\int_{\shell_{r+t}(x)\setminus \shell_r(x)}|\nabla u|^2\xi^2\,dz+C\int_{\shell_{r+t}(x)\setminus \shell_r(x)}u^2|\nabla\xi|^2\,dz\\
		&\le C\left(\frac{1}{t^2}+\sqrt{m}\right)\int_{\shell_{r+t}(x)\setminus \shell_r(x)}u^2\,dz+\frac12\int_{\shell_{r+t}(x)\setminus \shell_r(x)}|\nabla u|^2\xi^2\,dz.
		\end{aligned}
		\]
		thus
		\[
		\int_{\shell_{r+\frac{3t}{4}}(x)\setminus \shell_{r+\frac{t}{4}}(x)}|\nabla u|^2\,dz\le C\left(\frac{1}{t^2}+\sqrt{m}\right)\int_{\shell_{r+t}(x)\setminus \shell_r(x)}u^2\,dz,
		\]
so that	 \eqref{eq:caccioppoli}, and hence \eqref{eq:ee3}, holds. Combining~\eqref{eq:ee1}, \eqref{eq:ee2} and~\eqref{eq:ee3}, we obtain (for a  constant $C=C(Q)>0$)
	 \begin{equation}\label{eq:ee6}
	 \int_{\R^3}|\nabla u_1|^2+|\nabla u_2|^2-|\nabla u|^2\,dz\leq C\left(\frac{1}{t^2}+\sqrt{m}\right)(V(r+t)-V(r)).
	 \end{equation}
Now we notice that
	 \begin{equation}\label{eq:boundm}
	 	\int_\Omega |\nabla  u_1|^2+|\nabla u_2|^2\,dz\leq 2\int_\Omega|\nabla  u|^2+2\int_\Omega u^2(|\nabla f_1|^2+|\nabla f_2|^2)\leq C\Big(m^{3/2}+\frac{m}{t^2}\Big),
	 \end{equation}
where the second inequality follows by  minimality and estimate~\eqref{eq:boundm32grad} of Appendix~\ref{app2}.
	 
By~\eqref{eq:ee4}, \eqref{eq:ee6} and~\eqref{eq:boundm} we can finally estimate the gradient terms of~\eqref{eq:subadditivity}:
	 	\begin{equation}\label{eq:lastgrad}
	 	\begin{split}
	 	&\int_{\R^3}\overline C^2|\nabla  u_1|^2+\overline C^2|\nabla  u_2|^2-|\nabla  u|^2\,dz\\
	 	&\le \left(1+\frac{10}{m} \int_{\shell_{r+t}(x)\setminus \shell_r(x)}u^2\,dz\right)\left(\int_{\R^3}|\nabla  u_1|^2+ |\nabla  u_2|^2\,dz\right) -\int_{\R^3}|\nabla  u|^2\,dz \\
	 	&\leq C\left(\frac{1}{t^2}+\sqrt{m}\right)\int_{\shell_{r+t}(x)\setminus \shell_r(x)} u^2\,dz.
	 	\end{split}
	 	\end{equation}

{\bf Step 1.3: Analysis of the repulsive terms of~\eqref{eq:subadditivity}.}
We focus now the Coulombic terms. As 
in~\cite[equation~(3.6)]{LuOtto}, for $R>r+t$ one obtains
\begin{equation}\label{eq:e3}
	\begin{split}
		D( u_1^2, u_1^2)+D( u_2^2, u_2^2)-D( u^2, u^2)&\leq-2\int_{\R^3}\int_{\R^3}\frac{ u_1^2(z) u_2^2(y)}{|z-y|}\,dzdy\\
		&\leq -2\int_{\shell_r(x)}\int_{\shell_R(x)\setminus \shell_{r+t}(x)}\frac{ u_1^2(z) u_2^2(y)}{|z-y|}\,dzdy\\
		&=-2\int_{\shell_r(x)}\int_{\shell_R(x)\setminus \shell_{r+t}(x)}\frac{ u^2(z) u^2(y)}{|z-y|}\,dzdy\\
		&\leq -\frac{2}{{\mathrm{diam}(\shell)}R}\int_{\shell_r(x)}\int_{\shell_R(x)\setminus \shell_{r+t}(x)} u^2({z}) u^2(y)\,d{z}dy.
	\end{split}
\end{equation}
\noindent Moreover, using again the minimality and~\eqref{eq:boundm32grad}, we infer that for $i=1,2$
\begin{equation}\label{eq:boundDuu}
	D(u_i^2,u_i^2)\leq D(u^2,u^2)\le Cm^{3/2}.
\end{equation}
	Thus, recalling the estimate on $\overline C^2$ in~\eqref{eq:ee4}, by \eqref{eq:e3} and~\eqref{eq:boundDuu}, one obtains
	\begin{equation}\label{eq:lastD}
		\begin{aligned}
		&D(\overline C^2 u_1^2,\overline C^2 u_1^2)+D(\overline C^2 u_2^2,\overline C^2 u_2^2)-D( u^2, u^2)\\
		&\le \left(1+\frac{10}{m}\int_{\shell_{r+t}(x)\setminus \shell_r(x)} u^2\,dz\right)^2\Big(D( u_1^2, u_1^2)+D( u_2^2, u_2^2)\Big)-D( u^2, u^2)\\
		&\le 	-\frac{2}{{\mathrm{diam}(\shell)}R}\int_{\shell_r(x)}\int_{\shell_R(x)\setminus \shell_{r+t}(x)} u^2({z}) u^2(y)\,d{z}dy+\frac{C}{m}\bigg(\int_{\shell_{r+t}(x)\setminus \shell_r(x)} u^2\,dz\bigg)\Big(D(u_1^2,u_1^2)+D(u_2^2,u_2^2)\Big)\\
		&\le 	-\frac{2}{{\mathrm{diam}(\shell)}R}\int_{\shell_r(x)}\int_{\shell_R(x)\setminus \shell_{r+t}(x)} u^2({z}) u^2(y)\,d{z}dy+C\sqrt{m}\int_{\shell_{r+t}(x)\setminus \shell_r(x)} u^2\,dz.
\end{aligned} 
	\end{equation}

We can finally deduce the claim~\eqref{eq:V} by plugging~\eqref{eq:lastgrad} and~\eqref{eq:lastD} in~\eqref{eq:subadditivity}.

\vspace{.3cm} {\bf Step 2: The discrete ODE implies a mass concentration argument}

We now employ~\eqref{eq:V} to show that there exists a constant $C_0>0$ depending only on $\shell$ such that for all $m\geq 1$, for every $x\in\R^3$ and $R> 1$ the following implication holds
	\begin{equation}\label{eq:step2}
		V(x, R)\ge C_0\sqrt{m}\quad\Longrightarrow\quad V(x,2 R)\ge\frac{m}{2}.
	\end{equation}
For notational simplicity we write in this step $V(\cdot)$ in place of $V(x,\cdot)$.
Observe that~\eqref{eq:V}, with $4R$ in place of $R$, implies that for any $r\in[R,2R-t]$ there holds
\[
V(r+t)-V(r)\ge \frac{\gamma}{(\tfrac{1}{t^2}+\sqrt{m})4R}V(r)\left(V(4R)-V(r+t)\right).
\]
As $V(r+t)\le V(2R)$ and $V(r)\ge V(R)$ the previous inequality implies that
\[
V(r+t)-V(r)\ge \frac{\gamma}{(\tfrac{1}{t^2}+\sqrt{m})4R}V(R)\left(V(4R)-V(2R)\right).
\]
Now, we choose $t=1$ and $N=\lfloor R \rfloor$, so that the previous inequality implies that (using also that $m\geq 1$)
\[
\begin{aligned}
	V(2R)-V(R)&\geq\sum_{j=1}^N V\left(R+j\right) -V\left(R+j-1\right)\ge\frac{\gamma}{16\sqrt{m}}V(R)\left(V(4R)-V(2R)\right),
\end{aligned}
\]
so that
\[
V(4R)-V(2R)\le\frac{16\sqrt{m}}{\gamma V(R)}\left(V(2R)-V(R)\right).
\]
By rewriting the previous inequality with $2^{i}R$ in place of $R$, for $i\in \N\setminus\{0\}$, we deduce
\begin{equation}\label{eq:stimaVi}
V(2^{i+2}R)-V(2^{i+1}R)\le \frac{16\sqrt{m}}{\gamma  V(2^iR)}\left(V(2^{i+1}R)-V(2^{i}R)\right)
\le
\frac{16\sqrt{m}}{\gamma V(R)}\left(V(2^{i+1}R)-V(2^{i}R)\right)
\end{equation}
where we used that $V(R)\le V(2^{i}R)$. By iterating such an inequality one finally obtains a decay of the form
\[
V(2^{i+2}R)-V(2^{i+1}R)\le\left(\frac{16\sqrt{m}}{\gamma V(R)}\right)^{i+1}\left(V(2R)-V(R)\right).
\]
Now, for $k\in\N$, $k\ge1$ we deduce
\[
V(2^kR)-V(2R)=\sum_{i=2}^k V(2^iR)-V(2^{i-1}R)\le \sum_{i=2}^k\left(\frac{16\sqrt{m}}{\gamma V(R)}\right)^{i-1}\left(V(2R)-V(R)\right).
\]
Let us now fix $C_0=32/\gamma$. Then, if $V(R)\ge C_0 \sqrt{m}$, we have
\[
V(2^kR)-V(2R)\le V(2R)-V(R)\le V(2R).
\]
By sending $k\to+\infty$, we finally arrive to
\[
\frac{m}{2}\le V(2R).
\]
%%%%%%%%%%%%
%%%%%%%%%%%%%

\vspace{.3cm} {\bf Step 3: Conclusion by a covering argument}

Let $C_0$ be the constant from Step 2, $m\geq 1$ and set
\begin{equation}\label{eq: def Rm}
	R_m=\inf\left\{r > 1 \,:\, \exists x\in\R^3,\; V(x,r)\ge
	C_0m^{4/5}\right\}.
\end{equation}
	
	Since we are taking the infimum over a nonempty set, we infer that 
	$R_m<+\infty$. 
    We take $\widetilde{R}_m=R_m+\tfrac12$ and $\widetilde{x}\in \R^3$ with $V(\widetilde{x}, \widetilde{R}_m)\geq C_0m^{4/5}$.
    Using \eqref{eq:boundm32} from Appendix~\ref{app2}, we obtain 
   \begin{equation}\label{eq: Rm > 2}
        \widetilde C m^{3/2}\geq \F(\Omega)\geq D(u^2,u^2)\geq \int_{\shell_{\widetilde{R}_m}(\widetilde{x})}\int_{\shell_{\widetilde{R}_m}(\widetilde{x})}\frac{u^2(z)u^2(y)}{|z-y|}\,dzdy \geq \frac{C}{\widetilde{R}_m} m^{8/5},
   \end{equation}
   thus $\widetilde{R}_m\geq Cm^{1/10}$, so there exists a constant $\widetilde m(Q)$ such that $\widetilde{R}_m> 2$ as $m\geq \widetilde m$. 
Now, by Step 2, since $V(\widetilde{x}, \widetilde{R}_m)\geq C_0m^{4/5}\geq C_0\sqrt{m}$, we have
\[
\frac{m}{2}\le V(\widetilde{x},2\widetilde{R}_m).
\]
As $\shell_{2\widetilde{R}_m}(\widetilde{x})$ can be covered by a uniformly bounded number $n$ of $\shell_{\widetilde{R}_m/2}(x_i)$, for suitable centers $x_i\in \R^3$, we deduce
\[\shell_{2\widetilde{R}_m}(\widetilde{x})\subset\bigcup_{i=1}^{n} \shell_{\widetilde{R}_m/2}(x_i).
\]
Since $1<\widetilde{R}_m/2<R_m$, so that $V(\widetilde{x}, \widetilde{R}_m/2)\leq C_0m^{4/5}$, it follows 
\begin{equation}\label{eq:conclusion}
	\frac{m}{2} \leq V(\widetilde{x},2\widetilde{R}_m)\le \sum_{i=1}^{n}
	V(x_i,\widetilde{R}_m/2)\le nC_0 m^{4/5},
\end{equation}
and we fix $\overline{m}=\max\{(3n C_0)^5, \widetilde m\}$.
Finally, assuming $m \geq \overline{m}$, then claim \textit{i}) holds.

% Using \eqref{eq:boundm32}, we obtain 
%    \begin{equation}\label{eq: Rm > 2}
%         \widetilde C m^{3/2}\geq \F(\Omega)\geq D(u^2,u^2)\geq \int_{\shell_{R_m}(x)}\int_{\shell_{R_m}(x)}\frac{u^2(z)u^2(y)}{|z-y|}\,dzdy \geq \frac{C}{R_m} m^{8/5},
%    \end{equation}
%     thus $R_m\geq Cm^{1/10}$, so there exists a constant $\widetilde m(Q)$ such that $R_m> 2$ as $m\geq \widetilde m$.
% 	Now, by Step 2  we have
% 	\[
% 	\frac{m}{2}\le V(x,2R_m).
% 	\]
% As $\shell_{2R_m}(x)$ can be covered by a universally bounded number $n$ of $\shell_{R_m/2}$, we have
% $\shell_{2R_m}(x)\subset\cup_{i=1}^{n} \shell_{R_m/2}(x_i)$. By our choice of $R_m$, it follows that  
% \begin{equation}\label{eq:conclusion}
% 	\frac{m}{2} \leq V(x,2R_m)\le \sum_{i=1}^{n}
% 	V(x_i,R_m/2)\le nC_0 m^{4/5},
% \end{equation}
% and we fix $\overline{m}=\max\{(3n C_0)^5, \widetilde m\}$.
% Finally, assuming $m \geq \overline{m}$, then claim \textit{i}) holds.
\end{proof}

\begin{remark}
It is worth noting that the definition of $R_m$ in \eqref{eq: def Rm} requires $V(x,r)\ge C_0m^{4/5}$ instead of $V(x,r)\ge C_0\sqrt{m}$, which was the bound used in Step 2. This stronger assumption is necessary, since requiring only $V(x,r)\ge C_0\sqrt{m}$ in \eqref{eq: Rm > 2} is not enough to ensure that $\widetilde R_m > 2$.
\end{remark}

We are now in position to prove the main result of the paper, Theorem~\ref{thm:maintheorem}.
\begin{proof}[Proof of Theorem \ref{thm:maintheorem}]
\
By Lemma \ref{le:equivalence}, we can study problem \eqref{eq:mainnonex} with $q=m$. Let us consider $Q=[-1,1]^3$. Let $Q^i$, for $i\in \{1,2,3\}$, be the following cubes:
\[ Q^1=M_1Q,\quad Q^2=M_2Q,\quad Q^3=Q,\]
where $M_1,M_2$ are two rotations matrices chosen according to Lemma \ref{le: existence rotation}, satisfying \eqref{eq:detnonzero}.
By Theorem \ref{thm:nomin} with $Q=Q^i$, there exist $\overline{m}_i \ge 1$ for $i \in \{1, 2, 3\}$ for which a dichotomy holds true: either a minimizer does not exist, or an anomalous mass concentration occurs. Let us prove that there exists $\overline{m}\geq 1$ such that for all $m\geq \overline m$, problem \eqref{eq:mainnonex} does not admit a minimum. Let $\widehat{m}=\max\{\overline{m}_1,\overline{m}_2,\overline{m}_3\}$ and take $m\geq \widehat m$. We may assume that the problem admits a minimum for this choice of $m$, as otherwise the conclusion trivially holds. Let us take $u\in \mathcal{U}_m$, then by the dichotomy in Theorem \ref{thm:nomin}, there exist $x_i$ and $r_i$ such that
\begin{equation}\label{eq: strisce grandi sui cubi}
    \int_{Q^i_{r_i+1}(x_i)\setminus Q^i_{r_i}(x_i)} u^2dz\geq \frac{9}{10}m,\quad \forall i\in \{1,2,3\}.
\end{equation}
We define
\[ \widetilde Q^i=Q^i_{r_i+1}(x_i)\setminus Q^i_{r_i}(x_i), \]
then it holds
\begin{equation}
    m=\int_{\R^3}u^2\,dz\geq \int_{\widetilde Q^1} u^2 dz+\int_{\widetilde Q^2} u^2dz-\int_{\widetilde Q^1\cap \widetilde Q^2}u^2dz\geq \frac{9}{5}m-\int_{\widetilde Q^1\cap \widetilde Q^2}u^2dz,
\end{equation}
thus 
\[\int_{\widetilde Q^1\cap \widetilde Q^2}u^2dz\geq \frac{8}{10}m.\]
By the same computation using $\widetilde Q^1\cap \widetilde Q^2$ and $\widetilde Q^3$, we obtain
\begin{equation}
    m=\int_{\R^3}u^2\,dz\geq \int_{\widetilde Q^1\cap\widetilde Q^2 } u^2 dz+\int_{\widetilde Q^3} u^2dz-\int_{\widetilde Q^1\cap \widetilde Q^2\cap \widetilde Q^3}u^2dz\geq \frac{17}{10}m-\int_{\widetilde Q^1\cap \widetilde Q^2\cap \widetilde Q^3}u^2dz,
\end{equation}
and so
\[ \int_{\widetilde Q^1\cap \widetilde Q^2\cap \widetilde Q^3}u^2dz\geq\frac{7}{10}m.\]
Since the number of connected components of $\widetilde Q^1\cap \widetilde Q^2\cap \widetilde Q^3$ can be estimated from above by a dimensional constant $N$\footnote{Precisely $N=6^3$, the maximal number of possible intersections of sides of the three cubes.}, we obtain that there exists a connected component $A$ of $\widetilde Q^1\cap \widetilde Q^2\cap \widetilde Q^3$, with diameter uniformly bounded (with respect to $m$) by a universal constant $D$ (thanks to the choice of the rotations of the cubes, see Proposition \ref{prop:rotation}), such that
\[ \int_A u^2dz\geq \frac{7}{10N}m.\]
By the upper bound~\eqref{eq:boundm32} from Appendix~\ref{app2}, we obtain\[
\begin{split}
\widetilde Cm^{3/2}\geq \mathcal F_m(\Omega)\geq \int_{A}\int_{A}\frac{u^2(z)u^2(y)}{|z-y|}\,dzdy\geq\frac{1}{\text{diam}\, A} \left(\int_{A}u^2\,dz\right)^2\geq \frac{1}{D}\bigg(\frac{7}{10N}\bigg)^2 m^2,
\end{split}
\]
which implies that there exists a universal constant $\widetilde{m}$ such that $m\leq \widetilde m$. Then taking $\overline m=\max \{\widetilde m, \widehat m\}$, we conclude.
\end{proof}

\section{Asymptotic of the isoperimetric profile of the energy}\label{sec:asymptotic}
This section is devoted to the proof of Theorem \ref{thm:asymptotic profile}.

\begin{proof}[Proof of Theorem \ref{thm:asymptotic profile}]
Let $\rho = m^{-\frac{1}{4}}$  and define $\mathcal{Q} = \big\{Q_{\rho,i}\big\}_i$ where $Q_{\rho,i}$ are cubes of the type $\rho (Q + z_i)$ where $z_i\in \mathbb{Z}^3$ and $Q$ is the unit cube. Then set 
 \[\mathcal{Q}_{\le} := \left\{Q_\rho \in \mathcal{Q} : |\Omega \cap Q_\rho| \le \frac{1}{2}|Q_\rho|\right\}, \quad \mathcal{Q}_{>} := \left\{Q_\rho \in \mathcal{Q} : |\Omega \cap Q_\rho| > \frac{1}{2}|Q_\rho|\right\}.\]
 By~\eqref{eq:boundm32} from Appendix \ref{app2} there exists a universal $\widetilde C>0$ such that \[		\inf\left\{ \mathcal F_m(\Omega)\,:\, \Omega\subset
		\R^3,\;\text{open, with }|\Omega|=|B_1|\, \right\} \leq \widetilde Cm^{3/2}.\]
        Let us consider any open $\Omega$ such that $|\Omega|=|B_1|$ and $\F_m(\Omega)\leq 2\widetilde Cm^{3/2}$. By Proposition \ref{prop:existence optimal u}, there exists an optimal function $u$ for $\F_m(\Omega)$. We distinguish two different cases.

 \textbf{Case 1:}   $\sum\limits_{Q_\rho\in \mathcal{Q}_>}\int_{\Omega\cap Q_\rho} u^{2}\,dx \geq \frac{m}{2}.$
 
Set 
\begin{equation}\label{def:mu}
d\mu := \frac{u^2}{m}d\mathcal{L}^3
\end{equation}
and consider the quantity
\[
 I(\Omega) := \inf\left\{\int_\Omega\int_\Omega \frac{d\nu(x)d\nu(y)}{|x-y|},\; \nu(\Omega)=1\right\},
 \]
where the infimum is taken among positive Radon probability measures on $\Omega$.
So $\mu(\Omega)=1$ and $I(\Omega)=\frac{1}{\text{Cap}(\Omega)}$, where $\text{Cap}$ stands for the potential capacity, see \cite{landkof1972,LiebLoss}. Notice that $I$ is decreasing with respect to inclusion, i.e., if $\Omega\subset\Omega'$, then $I(\Omega)\ge I(\Omega')$. Then there holds
\begin{equation}\label{eq:C>}
\begin{aligned}
\frac{1}{m^{2}} D(u^{2},u^{2}) 
&= \int_{\Omega} \int_{\Omega} \frac{d\mu(x)d\mu(y)}{|x-y|} 
\ge \sum_{Q_\rho\in\mathcal{Q}_{>}} \int_{\Omega \cap Q_\rho} \int_{\Omega \cap Q_\rho} \frac{d\mu(x) d\mu(y)}{|x-y|} \frac{\mu(\Omega \cap Q_\rho)^{2}}{\mu(\Omega \cap Q_\rho)^{2}} \\
&\ge \sum_{Q_\rho \in \mathcal{Q}_{>}} \mu(\Omega \cap Q_\rho)^{2} I(\Omega \cap Q_\rho)
\ge \sum_{Q_\rho \in \mathcal{Q}_{>}} \mu(\Omega \cap Q_\rho)^{2} I(Q_\rho) 
= \frac{I(Q)}{\rho} \sum_{Q_\rho \in \mathcal{Q}_{>}} \mu(\Omega \cap Q_\rho)^{2},
\end{aligned}
\end{equation}
where we used the monotonicity of $I$ and the rescaling property $\text{Cap}(Q_\rho)=\rho \text{Cap}(Q)$. Since \[
\sum_{Q_\rho\in {\mathcal{Q}_{>}}} \int_{\Omega\cap Q_\rho} u^{2}\,dx \geq \frac{m}{2},
\]
then by the definition of $\mu$ in \eqref{def:mu} we have that $\sum\limits_{Q_\rho \in \mathcal{Q}_{>}} \mu(\Omega \cap Q_\rho) \geq \frac{1}{2}$ and this implies that
\begin{equation}\label{eq:C> 2}
    \sum_{Q_\rho \in \mathcal{Q}_{>}} \mu(\Omega \cap Q_\rho)^{2} \geq  \sum_{Q_\rho \in \mathcal{Q}_{>}}  \frac{\bigg(\sum\limits_{Q_\rho \in \mathcal{Q}_{>}} \mu(\Omega \cap Q_\rho)\bigg)^2}{(\#\mathcal{Q}_{>})^2}\geq \frac{\#\mathcal{Q}_{>}}{4(\#\mathcal{Q}_{>})^{2}}, 
\end{equation}
where $\# $ denotes the cardinality and we used that $\inf\left\{ \sum\limits_{i=1}^N a_i^2 \,:\, \sum\limits_{i=1}^N a_i= b  \right\}$ is attained by $a_i=\frac{b}{N}$. Thus by \eqref{eq:C>} and \eqref{eq:C> 2}, we obtain
\begin{equation}\label{eq:C> 3}
D(u^2,u^2)\geq \frac{I(Q)m^2}{\rho} \sum_{Q_\rho \in \mathcal{Q}_{>}} \mu(\Omega \cap Q_\rho)^{2} \geq  \frac{I(Q)m^2}{4\rho}\frac{1}{\#\mathcal{Q}_{>}}.
\end{equation}
Moreover we have an upper bound on the number of cubes in $\mathcal Q_>$ of the form $\#\mathcal{Q}_>\leq \frac{2|B_1|}{\rho^3}$, since $|\Omega|=|B_1|$, so we obtain
\[ D(u^2,u^2)\geq Cm^{3/2}, \]
where $C= \frac{I(Q)}{8|B_1|}$ and recalling that we fixed $\rho=m^{-\tfrac{1}{4}}$. 

\vspace{1em}
\noindent\textbf{Case 2:}   $\sum\limits_{\Omega\cap Q_\rho\in \mathcal{Q}_>}\int_{Q_\rho} u^{2}\,dx < \frac{m}{2}$. 

Since $\int_\Omega u^2 dx=m,$ then  $\sum\limits_{Q_\rho\in \mathcal{Q}_\leq}\int_{Q_\rho} u^{2}\,dx \geq \frac{m}{2}$. Then it holds
\begin{equation*}
\begin{aligned}
\int_{\Omega} |\nabla u|^{2} dx &\ge \sum_{Q_\rho \in \mathcal{Q}_{\leq}} \int_{Q_\rho } |\nabla u|^{2} \,dx\ge \sum_{Q_\rho \in \mathcal{Q}_{\leq}} \frac{\lambda_{1}^N( Q_\rho)}{6} \int_{Q_\rho } u^{2} \,dx \ge \frac{\lambda^N_{1}(Q_\rho) m}{12} = \frac{\lambda_1^N(Q)}{12\rho^{2}} m =\frac{\lambda_1^N(Q)}{12}m^{3/2},
\end{aligned}
\end{equation*}
where $\lambda^N_{1}(A)$ denotes the first nonzero Neumann eigenvalue of $A$ and the Poincar\'e inequality holds since if $Q_\rho\in\mathcal{Q}_\leq$, then $|\left\{u>0\right\} \cap Q_\rho| \leq |\Omega \cap Q_\rho|\le \frac{1}{2}|Q_\rho|$. In Appendix~\ref{app:poinc} we provide a proof of this well-known result for the sake of completeness, since we could not find an exact reference.

This implies that there exists a universal constant $\overline C$ such that, if $\F_m(\Omega)\leq 2\widetilde Cm^{3/2}$, then $\F_m(\Omega)\geq \overline Cm^{3/2}$.
\end{proof}

\begin{remark}[On the choice of $\rho=m^{-1/4}$]
Although the rigorous proof does not allow to select at once the fully charged and lightly charged cubes, one can guess that $\rho$ must be of order $m^{-1/4}$ since the balance of the two components of the energy is of order \[
m^2\rho^2+m\rho^{-2}.
\]
\end{remark}
	
\appendix
	
\section{Estimate for the functional on union of balls}\label{app2}
In this Appendix we provide an upper bound for the energy $\mathcal F_m$ of order $m^{3/2}$.  The computations are similar to those of~\cite[Section~6]{MazzoleniMuratovRuffini}, but we report them here for the reader's sake. We construct a competitor $\Omega_N$ made up of a suitably chosen quantity of disjoint balls with mutual distance diverging to infinity, for which we can explicitly estimate the energy.
	
Let $\Omega_N=\cup_{i=1}^NB_r(x_i)$, where \[
d=\min\{|x_i-x_j| : i\not=j\}\] 
is diverging to infinity. 
We select $N\in \N$ and $r>0$ so that $|\Omega_N|=|B_1|$; this implies in particular $Nr^3= 1$, that is $r= N^{-1/3}$.

Let $w_B\in H^1_0(B_1)$ be the first Dirichlet eigenfunction of $B_1$ normalized with $\int_{B_1}w_B^2(z)\,dz=m$. We define the test function $\widetilde w$ on $\Omega_N$ by\[
\widetilde w(z)=\sum_{i=1}^N w_B((z-x_i)/r),\qquad z\in \R^3,
\]
and clearly $\widetilde w\in H^1_0(\Omega_N)$.
Therefore (recalling that the balls $B_r(x_i)$ are pairwise disjoint), we have \[
\int_{\R^3}\widetilde w^2(z)\,dz=N\int_{B_r}w_B^2(z/r)\,dz=N\int_{B_1}w_B^2(y)r^3dy=\int_{B_1}w_B^2=m.
\]
Thus, by definition of $\mathcal F_m$,\[
\begin{split}
&\mathcal F_m(\Omega_N)\leq \mathcal F(\widetilde w)= N\int_{B_r}|\nabla w_B(z/r)|^2\,dz+\frac{N}{2}\int_{B_r}\int_{B_r}\frac{w_B^2(z/r)w_B^2(w/r)}{|z-w|}\,dzdw+o_d(1)\\
&=N^{2/3}\int_{B_1}|\nabla w_B|^2\,dy+\frac{N^{-2/3}}{2}D(w_B^2,w_B^2)+o_d(1)\leq \lambda_0(B_1)N^{2/3}m+\frac12(\sqrt{\lambda_0(B_1)}C_{univ}) N^{-2/3}m^2 +o_d(1),
\end{split}
\]
where $o_d(1)$ is a quantity that vanishes as the mutual distance between the centers of the balls diverges.
Minimizing with respect to $N$ gives the optimal choice $N=\lfloor Cm^{3/4} \rfloor$, which leads to\[
\mathcal F_m(\Omega_N)\leq C m^{3/2},
\]
for a universal constant $C>0$.

As a consequence, if $\Omega$ is an optimal set for problem~\eqref{eq:mainnonex}, then we have the bound from above
\begin{equation}\label{eq:boundm32}
\mathcal F_m(\Omega)\leq \widetilde C m^{3/2},
\end{equation}
for a universal $\widetilde C>0$.
We stress that, since both terms of the functional are nonnegative, we also deduce, calling $u$ an optimal function for $\mathcal F_m(\Omega)$, 
\begin{equation}\label{eq:boundm32grad}
\int_\Omega|\nabla u|^2\,dx\leq \widetilde C m^{3/2},\qquad \text{and}\qquad D(u^2,u^2)\leq \widetilde C m^{3/2}.
\end{equation}

\section{Existence of good rotations}\label{approt2}
\begin{lemma}\label{le: existence rotation}
    Let $X=SO(3)\times SO(3)$. Let  
    \[
Y=\bigcup_{(i,j,k)\in\{1,2,3\}}\left\{(M_1,M_2)\in X\,:\, \det({M_1}e_i,{M_2}e_j,e_k)=0  \right\}=\bigcup_{(i,j,k)\in\{1,2,3\}}Y_{i,j,k}.
    \]
Then $|Y|=\mathcal L^6(Y)=0$. In particular, there exists $(M_1,M_2)\in SO(3)\times SO(3)$ such that 
\begin{equation}\label{eq:detnonzero}
 \det({M_1}e_i,{M_2}e_j,e_k)\ne 0  \quad \forall \; i,j,k\in \{1,2,3\}.
\end{equation}
\end{lemma}
\begin{proof}
Clearly we just fix $i,j,k$ and show that $|Y_{i,j,k}|=0$ for each $(i,j,k)\in\{1,2,3\}$.  Let us define $f_{ijk}:X\to\R$ by
\begin{equation}\label{eq: det not 0}
  f_{ijk}(M_1,M_2)=  \det({M_1}{e}_i,{M_2}{e}_j, {e}_k).
\end{equation}
Since the determinant is a polynomial in the matrix entries, each $f_{ijk}$ is a real analytic function defined on the connected manifold $SO(3) \times SO(3)$. By the continuity of $f_{ijk}$, the preimage of $\{0\}$ is closed; hence, every $Y_{ijk}=f_{ijk}^{-1}(\{0\})$ is a closed subset of $SO(3) \times SO(3)$. Finally notice that $f_{ijk}$ is not identically $0$. As a consequence of the Unique Continuation Theorem, $|Y_{ijk}|=|f_{ijk}^{-1}(\{0\})|=0$, working in local charts, since the zero set of a nontrivial real-analytic function defined on an open set has measure zero (see for example~\cite{Mityagin}). 
\end{proof}

Let us define some notation for the following. Let $Q=[-1,1]^3$ and $Q_r(x)=\{ry+x: y\in \shell\}$ for $x\in \R^3$ and $r>0$. Let \[ \widehat Q^i=Q_{r_i+1}(x_i)\setminus Q_{r_i}(x_i) \]
where $x_i\in \R^3$ and $r_i>0$ for $i\in\{1,2,3\}.$

\begin{proposition}\label{prop:rotation}
      Let $\delta_0>0$, then there exists $C=C(\delta_0)>0$ such that for all $M_1,M_2\in SO(3)\times SO(3)$ satisfying
      \[
|\det(M_1 e_i, M_2e_j, e_k)|\ge  \delta_0,\quad  \forall\: i,j,k\in \{1,2,3\},
\]
        \begin{enumerate}
\item $|\widetilde P|\le C$;
\item the number of connected components of $\widetilde P$ is less than or equal to $6^3$;
\item the diameter of each connected component of $\widetilde P$ is bounded by $C$.
    \end{enumerate}
     where for brevity we set    \[
    \widetilde P =M_1\widehat Q^1\cap M_2\widehat Q^2\cap \widehat Q^3.
    \]
     Note that $C$ is in particular independent of $x_i$ and $r_i$, for $i=1,2,3$.
\end{proposition}
\begin{remark}
    We stress that there is an abuse of notation between the proof of Theorem \ref{thm:maintheorem} and the statement of Proposition \ref{prop:rotation}. Precisely: here we first define $\widehat{Q}^i$ as cubes with sides parallel to the axis and then rotate those, while in the proof of Theorem \ref{thm:maintheorem} we must first fix the rotations $M_i$ and then define $Q^i=M_iQ$. This is feasible since $C$ is independent of the cube centers $x_i$ and their radii  $r_i$. This choice is done in order to simplify the proof of Proposition \ref{prop:rotation}.
\end{remark}
\begin{proof}[Proof of Proposition \ref{prop:rotation}]
Let $M_1,M_2, \widetilde P$ be as in the statement. We define a generic infinite slab, for $t\in\R$ and $\nu\in\R^3$ as
\[
S(\nu,t)=\{t\le x\cdot \nu\le t+1 \}.
\]
Note that $\widetilde P$ is contained in the union of intersections of slabs of the form $S(R_ie_j,r_i)$ for $i\in \{1,2,3\}$ and $j\in \{1,\dots,6\}$ where $(e_1,e_2,e_3)$ is the canonical basis of $\R^3$, $e_{4}=-e_1, \ e_{5}=-e_2, \ e_{6}=-e_3$ and $R_1=M_1\in SO(3)$, $R_2=M_2\in SO(3)$ and $R_3=Id_{3\times 3}$. More precisely  $\widetilde P$ contains  at most $6^3$ connected components (obtained as the number of possible intersection between all of the slabs).
\[
|\widetilde P|\le \sum_{i,j,k\in \{1,\dots,6\}} |(x_1+S(R_1e_i,r_1))\cap (x_2+S(R_2e_j,r_2))\cap (x_3+S(R_3e_k,r_3))|.
\]
We show now that we  can bound the right-hand side of the previous formula by a suitable choice of $M_1$ and $M_2$. Let us define 
\[
P=P_{i,j,k}=(x_1+S(R_1e_i,r_1))\cap ({x_2+}S(R_2e_j,r_2))\cap ({x_3+}S(R_3e_k,r_3)),\qquad i,j,k\in\{1,\dots,6\}.
\]
Up to a translation (not renamed) of the {form $P+v$ for some $v\in \R^3$}, which does not change the volume, we may write
\[
\begin{aligned}
P&\subseteq \{-1\le x\cdot R_1e_i\le 1 \}\cap \{-1\le x\cdot R_2e_j\le 1 \}\cap \{-1\le x\cdot R_3e_k\le 1 \}\\
&=\{x\,:\, -(1,1,1)^T\le  A_{ijk}x\le (1,1,1)^T\}
\end{aligned}
\]
where the inequality is meant component-wise and where $A=A_{ijk}$ is the matrix with rows $A=(R_1 e_i;R_2e_j;R_3e_k)$. We are left to compute $|P|$. Since $AP\subseteq [-1,1]^3$, we obtain, 
\[
8=|[-1,1]^3|\geq|AP|=|P||\det(A)| 
\]
by the Area formula, so that
\[
|P|\leq\frac{8}{|\det(A)|}\le\frac{8}{\delta_0}.
\]

This proves items $(1)$ and $(2)$. To check point $(3)$ we observe that 
\[
P\subseteq A^{-1}[-1,1]^3=B[-1,1]^3,
\]
so that $y\in P$ implies that $y=Bx$ for some $x\in [-1,1]^3$. Letting $b_i$, $i=1,2,3$, be the rows of $B$, one obtains that
\[
\|y\|^2=\|Bx\|^2\le \sum_{i=1,2,3}|b_i\cdot x|^2\le3\sum_{i=1,2,3}|b_i|^2.
\]
Now, since for $i=1,2,3$, 
\[
\|B\|=\|A^{-1}\|=\frac{\|Adj(A)\|}{|\det A|}\le \frac{c}{\delta_0},
\]
we conclude that $\sum_{i=1,2,3}|b_i|^2\leq  \frac{C}{\delta_0^2}$.
This implies the last statement of the proposition and its proof.
\end{proof}

\section{A Poincar\'e-type inequality}\label{app:poinc}
In this appendix we prove the following Poincar\'e-type inequality.
\begin{theorem}\label{thm:poinc}
Let $\Omega\subset \R^3$ be an open set of finite measure, $\rho>0$, $Q_\rho=\rho Q$, where $Q$ is the unit cube and $u\in H^1_0(\Omega)$ (extended to zero outside $\Omega$) be such that $|\Omega\cap Q_\rho|\leq \tfrac12|Q_\rho|$. Then it holds
\[
 \int_{Q_\rho } |\nabla u|^{2}\,dx \ge \frac{\lambda_{1}^N( Q_\rho)}{6} \int_{Q_\rho } u^{2}\,dx
\]
where $\lambda_1^N(\cdot)$ denotes the first eigenvalue of the Neumann-Laplacian.
\end{theorem}
\begin{proof}
We set the average of $u$ on $Q_\rho$ as
\[
\bar u = \frac{1}{|Q_\rho|}\int_{Q_\rho} u\,dx.
\]
Then
\begin{align*}
  \int_{Q_\rho} u^2\,dx = \int_{Q_\rho} \bigl[(u-\bar u)+\bar u\bigr]^2 \,dx\leq 2\int_{Q_\rho} (u-\bar u)^2+ 2\int_{Q_\rho} \bar u^2 \,dx.
\end{align*}
We estimate the second term of the righ-hand side using the assumption
$|\Omega\cap Q_\rho|\leq \frac12 |Q_\rho|$. Since $u=0$ on
$Q_\rho\setminus\Omega$, one has
\begin{align*}
  \int_{Q_\rho} (u-\bar u)^2\,dx
  \geq \int_{Q_\rho\setminus\Omega} (u-\bar u)^2\,dx = \int_{Q_\rho\setminus\Omega} \bar u^2\,dx = |Q_\rho\setminus\Omega|\,\bar u^2 \geq \frac12 |Q_\rho|\,\bar u^2 = \frac12 \int_{Q_\rho} \bar u^2 \,dx.
\end{align*}
Hence
\begin{align*}
  \int_{Q_\rho} u^2\,dx
  \leq 2(1+2)\int_{Q_\rho}(u-\bar u)^2\,dx \leq \frac{6}{\lambda_1^N(Q_\rho)}
      \int_{Q_\rho}|\nabla u|^2\,dx,
\end{align*}
where in the last inequality we used the Poincar\'e-Wirtinger inequality, with optimal constant the first non-zero Neumann eigenvalue $\lambda_1^N(Q_\rho)$.
Equivalently,
\[
  \int_{Q_\rho}|\nabla u|^2\,dx
  \geq \frac{\mu_1(Q_\rho)}{6}\int_{Q_\rho}u^2\,dx.
\]
\end{proof}

\paragraph{\bf Acknowledgements} 
We warmly thank Cyrill B. Muratov for valuable discussions on the topic of this paper.
The authors have been partially supported by the MUR and the European Union via PRIN2022 project P2022R537CS.
The authors are members of
INdAM-GNAMPA and have been partially supported by the INdAM-GNAMPA projects CUP E5324001950001 and CUP E53C25002010001.  \\

%%%%%%%%%%%%%%%%%%%%%%%%%%%%%%%%%%%%%%%%%%%%%%%%%%%%%%%%%%%%%%%%%%%%%%%%%%%%%%%%%%%%%%%%%%%%%%%%%%%%%%%%%%%%%%%%%%%%%%%%%%%%%%%%%%%%%%%%%%%%%%%%%%%%%%%%%
%%%%%%%%%%%%%%%%%%%%%%%%%%%%%%%%%%%%%%%%%%%%%%%%%%%%%%%%%%%%%%%%%%%%%%%%%%%%%<<<<<<<<<<<<<<<<<<<BIBLIOGRAFIA>>>>>>>>>>>>>%%%%%%%%%%%%%%%%%%%%%%%%%%%%%%%%%%%%%%%%%%%%%%%%%%%%%%%%%%%%%%%%%%%%%%%%%%%%%%%%%%%%%%%%%%%%%%%%%%%%%%%%%%%%%
%%%%%%%%%%%%%%%%%%%%%%%%%%%%%%%%%%%%%%%%%%%%%%%%%%%%%%%%%%%%%%%%%%%%%%%%%%%%%%%%%%%%%%%%%%%%%%%%%%%%%%%%%%%%%%%%%%%%%%%%%%%%%%%%%%%%%%%%%%%%%%%%%%%%%%%%%
\bibliographystyle{abbrv}
\bibliography{bibliografia}	

@article {Mityagin,
    AUTHOR = {Mityagin, B. S.},
     TITLE = {The zero set of a real analytic function},
   JOURNAL = {Mat. Zametki},
  FJOURNAL = {Matematicheskie Zametki},
    VOLUME = {107},
      YEAR = {2020},
    NUMBER = {3},
     PAGES = {473--475},
      ISSN = {0025-567X,2305-2880},
   MRCLASS = {26E05 (26B10)},
  MRNUMBER = {4070868},
MRREVIEWER = {Anna\ Valette-Stasica},
       DOI = {10.4213/mzm12620},
       URL = {https://doi.org/10.4213/mzm12620},
}

@book{GilbargTrudinger,
	author = {Gilbarg, D. and Trudinger, N.S.},
	title = {Elliptic Partial Differential Equations of Second Order},
	journal = {Classics in Mathematics},
	volume = {224},
	publisher = {Springer Berlin, Heidelberg},
	year = {2001}
}

@book{LiebLoss,
	title={Analysis},
	author={Lieb, E.H. and Loss, M.},
	journal={Graduate Studies in Mathematics},
	volumne={14},
	year={2001},
	publisher={American Mathematical Society}
}

@article{BucurButtazzo,
  title = {On the characterization of the compact embedding of {S}obolev spaces},
  volume = {44},
  journal = {Calc. Var. Partial Differential Equations},
  author = {Bucur,  Dorin and Buttazzo,  Giuseppe},
  year = {2011},
  pages = {455-475}
}

@article{MazzoleniMuratovRuffini,
  title={An optimal design problem for a charge qubit},
  author={Mazzoleni, Dario and Muratov, Cyrill B and Ruffini, Berardo},
  journal={Comm. Partial Differential Equations},
  volume={50},
  pages={1029--1073},
  year={2025},
  publisher={Taylor \& Francis}
}

@article{ChoksiPeletier,
  title={Small volume fraction limit of the diblock copolymer problem: {I}. {S}harp-interface functional},
  author={Choksi, Rustum and Peletier, Mark A},
  journal={SIAM J. Math. Anal.},
  volume={42},
  number={3},
  pages={1334--1370},
  year={2010},
  publisher={SIAM}
}

@article{KM1,
  title={On an isoperimetric problem with a competing nonlocal term {I}: {T}he planar case},
  author={Kn{\"u}pfer, Hans and Muratov, Cyrill B},
  journal={Comm. Pure Appl. Math.},
  volume={66},
  number={7},
  pages={1129--1162},
  year={2013},
  publisher={Wiley Online Library}
}

@article{KM2,
  title={On an isoperimetric problem with a competing nonlocal term {II}: {T}he general case},
  author={Kn{\"u}pfer, Hans and Muratov, Cyrill B},
  journal={Comm. Pure Appl. Math.},
  volume={67},
  number={12},
  pages={1974--1994},
  year={2014},
  publisher={Wiley Online Library}
}

@book{Maggibook,
  title={Sets of Finite Perimeter and Geometric Variational Problems: An Introduction to Geometric Measure Theory},
  author={Maggi, Francesco},
  volume={135},
  series={Cambridge Studies in Advanced Mathematics},
  year={2012},
  publisher={Cambridge University Press},
  address={Cambridge}
}

@article {hartree,
    AUTHOR = {Hartree, D.},
     TITLE = {The wave mechanics of an atom with a non-{C}oulomb central field. {P}art {I}: {T}heory and methods},
   JOURNAL = {Math. Proc. Camb. Phil. Soc.},
    VOLUME = {24},
    NUMBER = {1},
      YEAR = {1928},
     PAGES = {89--110},
}

@article {lions,
    AUTHOR = {Lions, P.-L.},
     TITLE = {Solutions of {H}artree-{F}ock equations for {C}oulomb systems},
   JOURNAL = {Comm. Math. Phys.},
    VOLUME = {109},
    NUMBER = {1},
      YEAR = {1987},
     PAGES = {33--97},
}

@article {LuOtto,
    AUTHOR = {Lu, J. and Otto, F.},
     TITLE = {Nonexistence of a minimizer for {T}homas-{F}ermi-{D}irac-von {W}eizs\"{a}cker model},
   JOURNAL = {Comm. Pure Appl. Math.},
    VOLUME = {67},
    NUMBER = {10},
      YEAR = {2014},
     PAGES = {1605--1617},
}

@book {landkof1972,
    AUTHOR = {Landkof, N. S.},
     TITLE = {Foundations of modern potential theory},
    SERIES = {Die Grundlehren der Mathematischen Wissenschaften},
    VOLUME = {180},
 PUBLISHER = {Springer-Verlag, New York-Heidelberg},
      YEAR = {1972},
     PAGES = {x+424},
}

@article{BMP25,
      title={Non-spherical minimizers in the generalized liquid drop model for {Y}ukawa and truncated {C}oulomb potentials}, 
      author={Lia Bronsard and Benoît Merlet and Marc Pegon},
      year={2025},
      journal={arXiv:2510.11893},
      eprint={2510.11893},
      archivePrefix={arXiv},
      primaryClass={math.AP},
      url={https://arxiv.org/abs/2510.11893}, 
}

@article{FFMMM15,
  author  = {Figalli, Alessio and Fusco, Nicola and Maggi, Francesco and Millot, Vincent and Morini, Massimiliano},
  title   = {Isoperimetry and stability properties of balls with respect to nonlocal energies},
  journal = {Comm. Math. Phys.},
  year    = {2015},
  volume  = {336},
  number  = {1},
  pages   = {441--507},
  doi     = {10.1007/s00220-014-2244-1}
}

@article{FKN16,
  author  = {Frank, Rupert L. and Killip, Rowan and Nam, Phan Th{\`a}nh},
  title   = {Nonexistence of large nuclei in the liquid drop model},
  journal = {Letters in Mathematical Physics},
  year    = {2016},
  volume  = {106},
  number  = {8},
  pages   = {1033--1036},
  doi     = {10.1007/s11005-016-0860-2}
}

@article{FN21,
  author  = {Frank, Rupert L. and Nam, Phan Th{\`a}nh},
  title   = {Existence and nonexistence in the liquid drop model},
  journal = {Calc. Var. Partial Differential Equations},
  year    = {2021},
  volume  = {60},
  number  = {6},
  pages   = {223},
  doi     = {10.1007/s00526-021-02058-0}
}

@article{FNV18_CPAM,
  author  = {Frank, Rupert L. and Nam, Phan Th{\`a}nh and Van Den Bosch, Hanne},
  title   = {The ionization conjecture in {T}homas-{F}ermi-{D}irac-von {W}eizs{\"a}cker theory},
  journal = {Comm. Pure Appl. Math.},
  year    = {2018},
  volume  = {71},
  number  = {3},
  pages   = {577--614},
  doi     = {10.1002/cpa.21717}
}

@book{hen17,
  title={Shape optimization and spectral theory},
  editor={Henrot, Antoine},
  year={2017},
  publisher={De Gruyter Open}
}
\end{document}